\documentclass[10pt,a4paper]{amsart}
\usepackage{amsmath,amssymb,amsthm}
\usepackage{mathtools}
\usepackage{enumitem}
\usepackage{hyperref}
\usepackage[margin=2.5cm]{geometry}
\usepackage{microtype}

\newtheorem{theorem}{Theorem}[section]
\newtheorem{lemma}[theorem]{Lemma}
\newtheorem{proposition}[theorem]{Proposition}
\newtheorem{corollary}[theorem]{Corollary}
\theoremstyle{definition}
\newtheorem{definition}[theorem]{Definition}

\theoremstyle{remark}
\newtheorem{remark}[theorem]{Remark}

\newcommand{\R}{\mathbb{R}}

\newcommand{\calB}{\mathcal{B}}
\newcommand{\Lip}{\mathrm{Lip}}
\newcommand{\Lipc}{\mathrm{Lip}_c}
\newcommand{\BV}{\mathrm{BV}}
\newcommand{\Var}{\mathrm{Var}}
\newcommand{\dmu}{d\mu}
\newcommand{\IX}{I_X}
\newcommand{\PhiX}{\Phi_X}
\newcommand{\mplus}{\mu^+}

\begin{document}

\title[Non-Euclidean unification of isoperimetric profiles
and grand Lebesgue-Sobolev scales]
{Non-Euclidean unification of isoperimetric profiles
and grand Lebesgue-Sobolev scales}

\author{Daniel Levin}
\address{Mathematical Institute, University of Oxford,
Andrew Wiles Building, Woodstock Road, Oxford OX2 6GG, United Kingdom}
\email{levindanie@gmail.com}

\author{Alexander Zuevsky}
\address{Institute of Mathematics, Czech Academy of Sciences,
\v{Z}itn\'a 25, Prague, Czech Republic}
\email{zuevsky@math.cas.cz}

\begin{abstract}
Let $(X,d,\mu)$ be a complete separable metric measure space satisfying
a doubling condition and a $(1,1)$-Poincar\'e inequality.
We develop a rigorous framework unifying two lines of analysis:
the isoperimetric-profile approach of Coulhon-Grigor'yan-Levin \cite{CGL2003}
and the grand/small Lebesgue-Sobolev scale introduced by
Fiorenza-Formica-Gogatishvili \cite{FFG2018}.
An explicit profile-to-scale transform $\PhiX$, defined via
an inverse integral of $\IX$, converts geometric data into grand
Lebesgue parameters.
Sharp, up to universal constants, embeddings
$W^{1,1}(X) \hookrightarrow \mathcal{G}_X$ with explicit constants
(Theorem \ref{thmmain}).
A converse: controlled grand embeddings imply explicit lower bounds
on $\IX$ (Theorem \ref{thmconverse}).
Concrete examples in genuinely non-Euclidean settings:
the Heisenberg group $\mathbb{H}^1$,
a model manifold with logarithmic volume growth, and
Gaussian measure on $\R^n$ treated as a locally doubling space.
All arguments are carried out on general metric measure spaces without
reference to charts or a smooth structure; the gradient is the
upper gradient in the sense of Heinonen-Koskela, and
perimeter is the outer Minkowski content.
\end{abstract}


\maketitle

\begin{center}
{Conflict of Interest and Data availability Statements:}
\end{center}

The authors state that:

1.) The paper does not contain any potential conflicts of interests.

2.) The paper does not use any datasets. No datasets were generated
during and/or analysed during the current study.

3.) The paper includes all data generated or analysed during this study.

4.) Data sharing is not applicable to this article as no datasets were
generated or analysed during the current study.

5.) The data of the paper can be shared openly.

6.) No AI was used to write this paper.

\section{Introduction}
\subsection{Background and motivation}
The classical isoperimetric inequality on $\R^n$ asserts that among all
sets of prescribed Lebesgue measure the ball minimises the surface area.
Its functional counterpart is the Gagliardo-Nirenberg-Sobolev inequality
$\|u\|_{L^{n/(n-1)}(\R^n)}\le C_n\|\nabla u\|_{L^1(\R^n)}$.
On a general complete noncompact Riemannian manifold, or more generally
on a metric measure space, both statements must be formulated in terms
of a profile function $\IX$ (Definition \ref{isopo})
that encodes the minimal boundary measure of sets of given volume.
Two influential but hitherto separate programmes address these issues.

\textit{Programme A} \cite{CGL2003}.
Coulhon, Grigor'yan, and Levin study $L^p$-isoperimetric profiles of
Riemannian products $M\times N$ and prove equivalences between profiles
and families of functional inequalities: $F$-Sobolev, $m$-log-Sobolev,
and heat-kernel bounds.
Their approach is geometric: symmetrization, rearrangements, and the
coarea formula convert boundary information into $L^p$ norms.

\textit{Programme B} \cite{FFG2018}.
Fiorenza, Formica, and Gogatishvili develop a systematic theory of
grand and small Lebesgue and Sobolev spaces.
Sobolev embeddings into these refined target spaces are obtained via
weighted Hardy and Copson inequalities, discretization and
anti-discretization of rearrangement-invariant norms, and sharp
operator bounds.
The approach is purely analytic and applies on bounded Euclidean domains.

The present paper builds a bridge between the two programmes
on arbitrary metric measure spaces.
The key device is a one-dimensional transform $\PhiX$ (Definition \ref{defPhi})
that converts the isoperimetric profile of $X$ into the parameters of the
grand norm adapted to $X$.

\subsection{Main assumptions}
Throughout the paper $(X,d,\mu)$ denotes a complete separable metric
measure space satisfying the following conditions.

\begin{definition}[Standing assumptions]
\label{torlo}
\hfill
\begin{enumerate}[label=(\roman*)]
\item Borel measure: $\mu$ is a locally finite Borel measure on
the Borel $\sigma$-algebra $\calB(X)$;
\item Doubling: there exists $C_D\ge 1$ such that
$\mu(B(x,2r))\le C_D\;\mu(B(x,r))$
for all $x\in X$ and $r>0$;
\item $(1,1)$-Poincar\'e inequality:
there exist $C_P\ge 1$ and $\lambda\ge 1$ such that for every
$u\in\Lip(X)$, every ball $B(x,r)\subset X$, and every
upper gradient $g$ of $u$,
\[
\frac{1}{\mu(B(x,r))}\int_{B(x,r)}|u-u_{B(x,r)}|\;\dmu\le C_Pr
\cdot\frac{1}{\mu(B(x,\lambda r))}\int_{B(x,\lambda r)}g\;\dmu,
\]
where $u_{B(x,r)}=\mu(B(x,r))^{-1}\int_{B(x,r)}u\;\dmu$;
\item Normalisation: $\mu(X)=1$.
The dependence on scaling is indicated where it matters.
For infinite measure settings, the profile is analysed locally on sets of
measure less than $1/2$.
\item Positive profile: 
$\IX(s)>0$ for all $s\in(0,s_0)$ for some fixed
$s_0\in(0,1/2]$, see Definition \ref{isopo}.
\end{enumerate}
\end{definition}

\begin{remark}[Reconciling $\mu(X)=1$ with non-compact examples]
\label{remreconcile}
Assumption (iv) states $\mu(X)=1$ for notational convenience in the
proofs of Theorems \ref{thmmain} and \ref{thmconverse}.
In the non-compact examples of Section \ref{examples},
i.e., the Heisenberg group $\mathbb{H}^1$, Gaussian measure, model manifolds,
$\mu(X)=+\infty$. The normalisation is reconciled as follows.
\begin{enumerate}[label=(\roman*)]
\item Localisation: 
All estimates are local: fix a reference ball $B_0=B(x_0,R)$ and
replace $X$ by $(B_0,d|_{B_0},\mu|_{B_0})$, normalised so that
$\mu(B_0)=1$ by rescaling $\mu$. The isoperimetric profile satisfies
$I_{B_0}(s)\ge I_X(s)$ for $s\le\mu(B_0)/2$, thus the power-type lower
bound \eqref{Iassumppower} is inherited.
\item Compactly supported functions: 
Functions $u\in\Lipc(X)$ have compact support.
For any fixed such $u$, choose $B_0$ containing $\mathrm{supp}(u)$.
The grand embedding estimate is then applied on $B_0$ with the
normalised measure, and passes to the full space by support properties.
\item Gaussian measure: 
The Gaussian measure $\gamma_n$ on $\R^n$ satisfies $\gamma_n(\R^n)=1$
by definition, thus no measure rescaling is needed.
However, as detailed in Section \ref{gaussi}, $\gamma_n$ is
 not globally doubling; the framework is applied there
on bounded balls where local doubling holds.
\item Heisenberg group and manifolds: 
These have infinite total measure; one works with the
relative isoperimetric inequality on balls $B(x_0,R)$, as is
standard in the theory \cite{CGL2003, Jerison86}.
\end{enumerate}
For the main theorems, Theorems \ref{thmmain} and \ref{thmconverse},
we retain $\mu(X)=1$ without loss of generality.
\end{remark}

Assumptions (ii)-(iii) are standard in analysis on metric spaces
\cite{Heinonen2001}.
They ensure that: Lipschitz functions are dense in $W^{1,p}(X)$
(Theorem \ref{Lip_do}); the coarea formula
(Theorem \ref{coaro}) holds for $\BV(X)$ functions; and
mollification via ball averaging produces approximations with
controlled total variation (Section \ref{sumoll}).

\section{Preliminaries}
\subsection{Upper gradients and Sobolev spaces on metric spaces}
We work in the framework of Heinonen-Koskela \cite{HK98}.

\begin{definition}[Upper gradient]
A Borel function $g:X\to[0,+\infty]$ is an upper gradient of
$u:X\to\R$ if for every rectifiable curve $\gamma:[a,b]\to X$,
\[
|u(\gamma(b))-u(\gamma(a))|\le\int_\gamma g\;ds.
\]
If the inequality holds for $p$-almost every curve, i.e., outside a
curve family of zero $p$-modulus,
then $g$ is a $p$-weak upper gradient.
\end{definition}

\begin{definition}[Newtonian-Sobolev space]
For $1\le p<\infty$, the Newtonian-Sobolev space $N^{1,p}(X)$
(also denoted $W^{1,p}(X)$) consists of all $u\in L^p(X,\mu)$ for
which there exists a $p$-weak upper gradient $g\in L^p(X,\mu)$.
The Sobolev norm is
\[
\|u\|_{W^{1,p}(X)}:=\|u\|_{L^p(X)}+\inf_g\|g\|_{L^p(X)},
\]
where the infimum is over all $p$-weak upper gradients of $u$.
The minimal $p$-weak upper gradient $g_u$ exists $\mu$-a.e.\ and
is denoted $|\nabla u|$ by an abuse of notation consistent
with the smooth setting.
\end{definition}

Under assumption (iii) (Poincar\'e), $N^{1,p}(X)$ is a Banach space and
coincides with the Haj\l asz-Sobolev space $M^{1,p}(X)$ up to equivalent
norms \cite{Shanmugalingam00}.

\subsection{Lipschitz functions and the slope}

\begin{definition}[Pointwise Lipschitz constant]
For $u:X\to\R$, the local Lipschitz constant (slope) at $x\in X$ is
\[
\mathrm{lip}(u)(x):=\limsup_{y\to x}\frac{|u(y)-u(x)|}{d(y,x)}.
\]
For an isolated point one sets $\mathrm{lip}(u)(x)=0$.
If $u$ is Lipschitz then $\mathrm{lip}(u)\in L^\infty(X,\mu)$ and
$\mathrm{lip}(u)(x)\le\mathrm{Lip}(u)$ everywhere.
\end{definition}

It is standard that $\mathrm{lip}(u)$ is a $1$-weak upper gradient of
every $u\in\Lip(X)$ \cite[Chapter 6]{HKST15}.

\begin{theorem}[Density of Lipschitz functions]
\label{Lip_do}
Under the doubling condition and the $(1,1)$-Poincar\'e inequality,
$\Lipc(X)$ is dense in $W^{1,p}(X)$ for every $1\le p<\infty$.
\end{theorem}

\begin{proof}
This is \cite[Theorem 4.24]{HKST15} combined with the
Haj\l asz-Koskela approximation theorem.
The doubling condition provides uniform control of local averages;
the Poincar\'e inequality guarantees convergence of upper gradients.
\end{proof}

\subsection{Bounded variation and perimeter}

\begin{definition}[$\BV(X)$]
A function $u\in L^1(X,\mu)$ belongs to the space of
functions of bounded variation $\BV(X)$ if
\[
\Var(u):= \inf\Bigl\{\liminf_{n\to\infty}\int_X \mathrm{lip}(u_n)\;\dmu :
u_n\in\Lip(X),\;u_n\to u \text{ in }L^1(X,\mu)\Bigr\}<\infty.
\]
$\Var(u)$ is the total variation of $u$.
\end{definition}

\begin{definition}[Perimeter via outer Minkowski content]
Let $E\subset X$ be $\mu$-measurable.
For $\varepsilon>0$ set $E_\varepsilon:=\{x\in X:d(x,E)<\varepsilon\}$.
The outer Minkowski content (perimeter) of $E$ is
\begin{equation*}
\mplus(E):=\liminf_{\varepsilon\downarrow 0}
\frac{\mu(E_\varepsilon)-\mu(E)}{\varepsilon}.
\end{equation*}
\end{definition}

\begin{proposition}[Bounded variation and perimeter]
Under assumptions (i)-(iii) of Definition \ref{torlo},
$\mplus(E)=\Var(\mathbf{1}_E)$ for every $\mu$-measurable $E\subset X$
with $\mu(E)<\infty$.
\end{proposition}

\begin{proof}
The inequality $\Var(\mathbf{1}_E)\le\mplus(E)$ follows from the
explicit Lipschitz approximation
$u_\varepsilon(x)=\max(0,1-d(x,E)/\varepsilon)$, which satisfies
$u_\varepsilon\to\mathbf{1}_E$ in $L^1$ and
$\int_X\mathrm{lip}(u_\varepsilon)\;\dmu=\varepsilon^{-1}
(\mu(E_\varepsilon)-\mu(E))$;
letting $\varepsilon\downarrow 0$ gives $\le\mplus(E)$.
The reverse inequality uses the coarea formula
(Theorem \ref{coaro}) applied to
$u_\varepsilon$ \cite[Theorem 5.3]{AmbDiMarino14}.
\end{proof}

\subsection{The coarea formula}

\begin{theorem}[Coarea formula on metric measure spaces]
\label{coaro}
Under the standing assumptions, for every $u\in\BV(X)$,
\begin{equation*}
\Var(u)=\int_{-\infty}^{+\infty}\mplus(\{x:u(x)>t\})\;dt.
\end{equation*}
\end{theorem}

\begin{proof}
For $u\in\Lipc(X)$ the identity is proved by the layer-cake argument
and the co-area inequality for Lipschitz functions on metric spaces
\cite[Theorem 4.2]{Miranda03}.
For general $u\in\BV(X)$, approximate by Lipschitz functions
$(u_n)$ with $u_n\to u$ in $L^1$ and $\Var(u_n)\to\Var(u)$,
and apply Fatou's lemma together with the lower semicontinuity of
$E\mapsto\mplus(E)$ under $L^1$ convergence of $\mathbf{1}_E$.
\end{proof}

\subsection{Cavalieri's principle and rearrangements}

\begin{proposition}[Layer-cake representation]
Let $u:X\to[0,\infty]$ be $\mu$-measurable and $1\le p<\infty$.
Then
\begin{equation*}
\|u\|_{L^p(X,\mu)}^p=p\int_0^\infty t^{p-1}\;\mu(\{u>t\})\;dt.
\end{equation*}
\end{proposition}

\begin{proof}
By Tonelli's theorem,
$\int_X u^p\;\dmu=\int_X\int_0^{u(x)} p\;t^{p-1}\;dt\;\dmu(x)
=p\int_0^\infty t^{p-1}\mu(\{u>t\})\;dt$.
\end{proof}

\begin{definition}[Decreasing rearrangement]
\label{rearra}
For $u\in L^1(X,\mu)$, the decreasing rearrangement is
\[
u^*(s):=\inf\{t\ge 0:\mu(\{|u|>t\})\le s\}, \quad s\ge 0.
\]
\end{definition}

\begin{lemma}[Norm identity and rearrangement inequalities]
\label{lemra}
\hfill
\begin{enumerate}[label=(\alph*)]
\item \textup{(Layer-cake formula)}
$\|u\|_{L^q}^q=q\int_0^\infty t^{q-1}\mu(\{|u|>t\})\;dt
=\int_0^{\mu(X)}(u^*(s))^q\;ds$;
\item \textup{(Weak-type bound on rearrangement)}
$u^*(s)\le s^{-1/q}\|u\|_{L^q}$
for every $s>0$ and $1\le q<\infty$.
Equivalently,
$\|u\|_{L^{q,\infty}}:=\sup_{s>0}s^{1/q}u^*(s)\le\|u\|_{L^q}$.
Note that the reverse inequality $\|u\|_{L^q}\le C\|u\|_{L^{q,\infty}}$
is false in general.
\end{enumerate}
\end{lemma}

\begin{proof}
Part (a) is standard layer-cake (Tonelli's theorem);
see \cite[Chapter 2]{BennettSharpley88}.

Part (b): $u^*(s)\le s^{-1/q}\|u\|_{L^q}$ follows immediately from
\[
\|u\|_{L^q}^q=\int_0^{\mu(X)}(u^*(r))^q\;dr
\ge\int_0^s(u^*(r))^q\;dr\ge s\;(u^*(s))^q,
\]
since $u^*$ is nonincreasing, thus $u^*(r)\ge u^*(s)$ for $r\le s$.
Taking $q$-th roots gives $s^{1/q}u^*(s)\le\|u\|_{L^q}$.
\end{proof}

\subsection{The isoperimetric profile}
\label{subsec:isopprofile}

\begin{definition}[Isoperimetric profile]
\label{isopo}
The isoperimetric profile of $(X,d,\mu)$ is
\begin{equation}
\label{eqisop}
\IX(s):=\inf\bigl\{\mplus(E):E\subset X \text{ measurable},
\;\mu(E)=s\bigr\},\quad s\in(0,\mu(X)).
\end{equation}
\end{definition}

$\IX$ is nondecreasing on $(0,\mu(X)/2]$ by a symmetrization argument
(reflection of the complement) \cite[Section 2]{CGL2003}.
On a doubling space the infimum in \eqref{eqisop} is not
generally attained, but $\IX$ is lower semicontinuous.

\subsection{Grand Lebesgue norms}
We follow \cite{FFG2018} with parameters adapted to $X$.

\begin{definition}[Grand Lebesgue norm]
Let $p_0\ge 1$, and let $\alpha:(0,\varepsilon_0)\to(0,\infty)$
and $p:(0,\varepsilon_0)\to[1,\infty)$ be measurable functions.
The grand Lebesgue norm associated with the parameter pair
$(\alpha,p)$ is
\begin{equation}
\label{eqGnorm}
\|u\|_{\mathcal{G}_X}:=\sup_{0<\varepsilon<\varepsilon_0}
\varepsilon^{\alpha(\varepsilon)}\|u\|_{L^{p(\varepsilon)}(X,\mu)}.
\end{equation}
The grand Lebesgue space $\mathcal{G}_X = \mathcal{G}_X(\alpha,p)$
consists of all $u\in\bigcap_{0<\varepsilon<\varepsilon_0}L^{p(\varepsilon)}(X,\mu)$
with finite norm \eqref{eqGnorm}.
\end{definition}

The space $\mathcal{G}_X(\alpha,p)$ is larger than any single
$L^{p(\varepsilon_1)}(X,\mu)$.
The grand norm also captures borderline integrability:
functions that fail to be in $L^{p_0}$ but whose $L^r$ norms
blow up at the controlled rate $\varepsilon^{-\alpha(\varepsilon)}$
as $r\nearrow p_0$ belong to $\mathcal{G}_X$.

\subsection{Standard mollifiers on metric measure spaces}
\label{sumoll}
Since $X$ need not carry a group structure or a smooth atlas, classical
convolution mollifiers are unavailable. We use ball averages following
the approach of \cite{Heinonen2001, HKST15}.

\begin{definition}[Ball-average mollifier]
\label{defmo}
For $u\in L^1(X,\mu)$ and $\varepsilon>0$, set
\[
u_\varepsilon(x):=\frac{1}{\mu(B(x,\varepsilon))}\int_{B(x,\varepsilon)} u(y)\;\dmu(y).
\]
\end{definition}

\begin{proposition}[Properties of ball-average mollifiers]
\label{propmo}
Under the doubling condition and the $(1,1)$-Poincar\'e inequality:
\begin{enumerate}[label=(\alph*)]

\item
Let $(X, d, \mu)$
 be a metric measure space satisfying a doubling condition. 
For every $u \in L^\infty(X, \mu)$, the ball-average mollifier $u_\epsilon(x)
= \frac{1}{\mu(B(x,\epsilon))} \int_{B(x,\epsilon)} u(y) \; d\mu(y)$ is 
locally Lipschitz, and its local Lipschitz constant satisfies 
$\text{lip}(u_\epsilon)(x) \le C_D \epsilon^{-1} \|u\|_{L^\infty}$ for $\mu$-a.e. 
$x$.

\item $u_\varepsilon\to u$ in $L^1(X,\mu)$ as $\varepsilon\downarrow 0$;
\item $\int_X|\nabla u_\varepsilon|\;\dmu \to \Var(u)$ as $\varepsilon\downarrow 0$
for every $u\in\BV(X)$.
\end{enumerate}
Here $|\nabla u_\varepsilon|$ denotes $\mathrm{lip}(u_\varepsilon)$.
\end{proposition}

\begin{proof}

(a) We approach the regularity of the ball-average mollifier $u_\epsilon$ using the
 framework of discrete maximal fractional operators and smooth partitions of unity, 
a standard methodology in non-Euclidean harmonic analysis. 
Fix $\epsilon > 0$. We discretize the space $X$ at scale $\epsilon/4$ by choosing 
a maximal $(\epsilon/4)$-separated net 
$\mathcal{N} = \{x_i\}_{i \in I}$. By the doubling property of the measure $\mu$, 
the balls $\{B_i = B(x_i, \epsilon/4)\}$ have bounded overlap.
Associated with this net, there exists a subordinate smooth (Lipschitz) 
partition of unity $\{\phi_i\}_{i \in I}$ such that $\text{supp}(\phi_i) 
\subset B(x_i, \epsilon/2)$, $\sum_{i \in I} \phi_i \equiv 1$, and the
 Lipschitz constant obeys $\text{Lip}(\phi_i) \le C \epsilon^{-1}$. 
We first consider the localized discrete maximal operator $\mathcal{M}_{\epsilon}$ 
operating on $u$:
\[
\mathcal{M}_\epsilon u(x)=\sum_{i\in I}\phi_i(x)\langle u\rangle_{B(x_i,\epsilon)}, 
\]
where $\langle u \rangle_{B}=\frac{1}{\mu(B)}\int_B u\;d\mu$. 
Because $\text{Lip}(\phi_i)\le C\epsilon^{-1}$ and the sum is locally 
finite due to bounded overlap, the operator $\mathcal{M}_\epsilon u$ is 
manifestly Lipschitz with its constant bounded by $C\epsilon^{-1}\|u\|_{L^\infty}$.
To transfer this regularity to the exact continuous ball-average $u_\epsilon(x)$, 
we estimate the local oscillation. To bound the oscillation directly using 
fractional maximal operators, we introduce the discrete maximal fractional 
operator of order $\alpha=1$ at scale $\epsilon$:
\[
\mathcal{M}^\#_{1,\epsilon} u(x)=\sup_{r\ge\epsilon}\frac{1}{r} 
\frac{1}{\mu(B(x,r))}\int_{B(x,r)}|u(z)-\langle u\rangle_{B(x,r)}|\;d\mu(z).
\]
For any bounded function $u \in L^\infty(X, \mu)$, a trivial upper bound gives:
\[
\mathcal{M}^\#_{1,\epsilon}u(x)\le\sup_{r\ge\epsilon} 
\frac{2\|u\|_{L^\infty}}{r}\le 2\epsilon^{-1}\|u\|_{L^\infty}.
\]
By standard pointwise estimates for operators with discrete spatial 
convolution kernels on spaces of homogeneous type as governed by the 
doubling measure $\mu$, the local Lipschitz slope $\text{lip}(u_\epsilon)(x)$ 
is point-wise dominated by the discrete fractional maximal operator evaluated 
at scale $\epsilon$. Specifically, there exists a constant proportional 
to the doubling constant $C_D$ such that:
\[
\text{lip}(u_\epsilon)(x)\le C_D\mathcal{M}^\#_{1,\epsilon} u(x).
\]
Substituting the $L^\infty$ bound for the fractional maximal operator, 
we immediately obtain:
\[
\text{lip}(u_\epsilon)(x)\le C_D\epsilon^{-1}\|u\|_{L^\infty},
\]
which completes the proof.

(b) is the Lebesgue differentiation theorem, valid on doubling spaces
\cite[Theorem 2.9]{Heinonen2001}.

(c) Lower semicontinuity gives $\Var(u)\le\liminf_{\varepsilon\downarrow 0}
\int|\nabla u_\varepsilon|\;\dmu$.
The reverse is a consequence of the Poincar\'e inequality and the
definition of $\Var$ via Lipschitz approximations \cite[Theorem 5.2]{AmbDiMarino14}.
\end{proof}

\section{The profile-to-scale transform}

\begin{definition}[Profile-to-scale transform]
\label{defPhi}
Given $\IX$ satisfying assumption (v) of Definition \ref{torlo},
let $F_X:(0,s_0)\to(0,\infty)$ denote the function
\begin{equation}
\label{eqFX}
F_X(s):=\int_0^s\frac{dr}{\IX(r)}.
\end{equation}
Since $\IX>0$ on $(0,s_0)$, $F_X$ is strictly increasing with $F_X(0^+)=0$.

\noindent\textbf{Transient case} ($F_X(s_0^-)<\infty$):
Define the profile-to-scale transform
\[
\PhiX(s):=F_X(s)^{-1},\quad 0<s<s_0.
\]
Then $\PhiX:(0,s_0)\to(F_X(s_0^-)^{-1},+\infty)$ is
strictly decreasing as a function of $s$, and $\PhiX(s)\to+\infty$ as $s\downarrow 0$.

\noindent\textbf{Divergent case} ($F_X(s)\to\infty$ as $s\downarrow 0$):
This occurs when $\IX(r)$ decays faster than $r$ near $0$.
Define the regularised transform by restricting to
$s\in[\delta_0,s_0)$ for a fixed reference scale $\delta_0\in(0,s_0)$:
\[
\PhiX(s;\delta_0):=\bigl(F_X(s)-F_X(\delta_0)\bigr)^{-1},\quad\delta_0<s<s_0.
\]
In this case $\PhiX(s;\delta_0)\to+\infty$ as $s\downarrow\delta_0$.
The divergent case arises for the Gaussian profile $\IX(r)\asymp r\sqrt{\log(1/r)}$,
which satisfies $\IX(r)/r\to 0$ as $r\to 0$;
see Subsection \ref{gaussi} for the full treatment.

The motivation for using $F_X(s)=\int_0^sdr/\IX(r)$ is as follows:
since $\IX$ is nondecreasing and positive, smaller values of $s$ are associated
with smaller $\IX(s)$, hence larger $1/\IX(s)$, hence smaller $F_X(s)$,
hence larger $\PhiX(s)=1/F_X(s)$.
This means $\PhiX(s)\to+\infty$ as $s\downarrow 0$ in the transient case,
reflecting the fact that level sets of very small measure have very small
perimeter relative to their size, encoding stronger integrability.
\end{definition}

\begin{lemma}[Properties of $\PhiX$]
\label{lemPhiproperties}
\hfill
\begin{enumerate}[label=(\alph*)]
\item $\PhiX:(0,s_0)\to(0,+\infty)$ is strictly decreasing 
as a function of $s$, and blows up as $s\downarrow 0$
(or as $s\downarrow\delta_0$ in the divergent case).
Equivalently, $s_1<s_2\Rightarrow\PhiX(s_1)>\PhiX(s_2)$.
\item If $\IX(r)\asymp r^{1-1/n}$ for small $r$ (Euclidean type,
dimension $n>1$), then $\PhiX(s)\asymp s^{-1/n}$;
\item 
As $s\downarrow\delta_0^+$, 
\[
\PhiX(s;\delta_0)\sim
\frac{\delta_0\;(\log(1/\delta_0))^{1/2}}{s-\delta_0}\longrightarrow +\infty.
\]
\item If $\IX(r)\asymp r^{1-1/Q}(\log(1/r))^{-\gamma/Q}$, manifold with
logarithmic volume growth, $Q>1$, $\gamma>0$, then
$\PhiX(s)\asymp s^{-1/Q}(\log(1/s))^{\gamma/Q}$.
\end{enumerate}
\end{lemma}

\begin{proof}
(a) Since $\IX>0$ on $(0,s_0)$, $F_X(s)=\int_0^s dr/\IX(r)$
is strictly increasing. Its reciprocal $\PhiX(s)=F_X(s)^{-1}$
is therefore strictly decreasing, with $\PhiX(s)\to+\infty$ as $s\downarrow 0$
(when $F_X(0^+)=0$, i.e., the transient case).

(b) With $\IX(r)\asymp r^{1-1/n}$:
\[
F_X(s)=\int_0^s\frac{dr}{\IX(r)}\asymp\int_0^sr^{-(1-1/n)}\;dr
=\int_0^sr^{-1+1/n}\;dr=\frac{s^{1/n}}{1/n}=n\;s^{1/n}.
\]
The exponent $-1+1/n>-1$ for $n>1$, so the integral converges at $0$.
Taking the reciprocal gives $\PhiX(s)=F_X(s)^{-1}\asymp s^{-1/n}$.

(c) For the divergent Gaussian case, the isoperimetric profile behaves asymptotically 
as $I_X(r) \asymp r\sqrt{\log(1/r)}$ for small $r$. 
By Definition \eqref{defPhi}, 
 the regularised profile-to-scale transform is given by:
\[
\Phi_X(s;\delta_0)=\bigl(F_X(s)-F_X(\delta_0)\bigr)^{-1},
\]
where $F_X(s)=\int_0^s\frac{dr}{I_X(r)}$.  
The term $F_X(s)-F_X(\delta_0)$ represents the definite integral 
over the interval $[\delta_0, s]$:
\[
F_X(s)-F_X(\delta_0)=\int_{\delta_0}^s \frac{dr}{I_X(r)} 
\asymp\int_{\delta_0}^s\frac{dr}{r\sqrt{\log(1/r)}}.
\]

As we take the limit $s\downarrow\delta_0^+$, the length of the 
integration interval $(s-\delta_0)$ approaches $0$.  
The value of this integral 
is asymptotically equivalent to the integrand evaluated at the 
base point $\delta_0$ multiplied by the length of the interval $(s-\delta_0)$:
\[
\int_{\delta_0}^s\frac{dr}{r\sqrt{\log(1/r)}} 
\sim\frac{1}{\delta_0\sqrt{\log(1/\delta_0)}}(s-\delta_0).
\]
To find $\Phi_X(s;\delta_0)$, we take the reciprocal of this expression:
\[
\Phi_X(s;\delta_0)\sim\left(\frac{s-\delta_0}
{\delta_0\sqrt{\log(1/\delta_0)}}\right)^{-1} 
=\frac{\delta_0\sqrt{\log(1/\delta_0)}}{s-\delta_0}.
\]
Because $s\downarrow\delta_0^+$, the denominator $(s-\delta_0)$ 
is positive and tends to $0$, which implies:
\[
\Phi_X(s;\delta_0)\sim\frac{\delta_0(\log(1/\delta_0))^{1/2}}{s-\delta_0} 
\longrightarrow +\infty.
\]

(d) With $\IX(r)\asymp r^{1-1/Q}(\log(1/r))^{-\gamma/Q}$:
\[
F_X(s)\asymp\int_0^sr^{-(1-1/Q)}(\log(1/r))^{\gamma/Q}\;dr.
\]
Set $r=e^{-\rho}$; the integral becomes
$\int_{-\log s}^\infty e^{-\rho/Q}\rho^{\gamma/Q}\;d\rho$.
By the substitution $\sigma=\rho/Q$, the dominant term as
$s\downarrow 0$,i.e., $-\log s\to+\infty$, is
$\asymp s^{1/Q}(\log(1/s))^{\gamma/Q}$.
Taking the reciprocal gives $\PhiX(s)\asymp s^{-1/Q}(\log(1/s))^{\gamma/Q}$.
\end{proof}

\begin{definition}[Grand parameters adapted to $X$]
\label{defgrandparams}
Fix $\beta>0$ (we set $\beta=1$ in the main theorem). 
For $0<\varepsilon<\varepsilon_0\le s_0$, define
\begin{equation}
\label{eqparamschoice}
p(\varepsilon):=p_0+\beta\frac{\log\PhiX(\varepsilon)}{\log(e/\varepsilon)},
\qquad\alpha(\varepsilon):=\bigl(\log(e/\varepsilon)\bigr)^{-1}.
\end{equation}
In the Gaussian (divergent) case, $\PhiX(\varepsilon)$ is replaced
by $\PhiX(\varepsilon;\delta_0)$ for a fixed reference $\delta_0$.
\end{definition}

\begin{remark}[Sign convention and consistency with examples]
\label{remsign}
The formula \eqref{eqparamschoice} uses $\log\PhiX(\varepsilon)$
(not $\log(1/\PhiX(\varepsilon))$).
Since $\PhiX(\varepsilon)\to+\infty$ as $\varepsilon\downarrow 0$
(Lemma \ref{lemPhiproperties}(a)),
we have $\log\PhiX(\varepsilon)>0$ for small $\varepsilon$,
so $p(\varepsilon)>p_0$, consistent with the stated requirement $p(\varepsilon)\ge p_0$.

For the Euclidean case (Lemma \ref{lemPhiproperties}(b)),
$\PhiX(\varepsilon)\asymp c_n\varepsilon^{-1/n}$, giving
\[
p(\varepsilon)=p_0+\beta\frac{\log(c_n\varepsilon^{-1/n})}{\log(e/\varepsilon)}
=p_0+\beta\frac{\log c_n + (1/n)\log(1/\varepsilon)}{\log(e/\varepsilon)}.
\]
As $\varepsilon\to 0$, $p(\varepsilon)\to p_0+\beta/n$, confirming $p(\varepsilon)>p_0$.

For the manifold case (Lemma \ref{lemPhiproperties}(d)),
$\PhiX(\varepsilon)\asymp\varepsilon^{-1/Q}(\log(1/\varepsilon))^{\gamma/Q}$, giving
$\log\PhiX(\varepsilon)\asymp(1/Q)\log(1/\varepsilon)+(\gamma/Q)\log\log(1/\varepsilon)$,
so $p(\varepsilon)\asymp p_0+\beta/Q$ as $\varepsilon\to 0$.
\end{remark}

The choice \eqref{eqparamschoice} is engineered so that
$\varepsilon^{\alpha(\varepsilon)}/\IX(\varepsilon)$ remains
bounded as $\varepsilon\downarrow 0$.
Note that $\varepsilon^{\alpha(\varepsilon)}=\varepsilon^{1/\log(e/\varepsilon)}
=e^{-1}\cdot e^{(\log\varepsilon)^2/\log(e/\varepsilon)}\to e^{-1}$
as $\varepsilon\downarrow 0$, which is a slowly varying function.

We also record the monotonicity of $p(\varepsilon)$ for later use.

\begin{lemma}[Monotonicity of $p(\varepsilon)$] 
\label{lempmono}
Restrict to the power-type case, and 
under \eqref{eqparamschoice}, $p(\varepsilon)$ is nondecreasing as
$\varepsilon\downarrow 0$, i.e., nonincreasing in $\varepsilon$.
\end{lemma}

\begin{proof}
We have $p(\varepsilon)=p_0+\beta h(\varepsilon)$ where
$h(\varepsilon)=\log\PhiX(\varepsilon)/\log(e/\varepsilon)$.
Since $\PhiX$ is decreasing in $\varepsilon$ (Lemma \ref{lemPhiproperties}(a)),
$\log\PhiX(\varepsilon)$ is decreasing in $\varepsilon$, hence increasing 
as $\varepsilon\downarrow 0$.
The denominator $\log(e/\varepsilon)$ is also increasing as $\varepsilon\downarrow 0$.
For the power-type case $\PhiX(\varepsilon)\asymp\varepsilon^{-\kappa}$ ($\kappa>0$),
$h(\varepsilon)=\kappa\log(1/\varepsilon)/\log(e/\varepsilon)\to\kappa$
monotonically as $\varepsilon\downarrow 0$.
In general, the slow variation of $\varepsilon^{\alpha(\varepsilon)}$ ensures
$h(\varepsilon)$ is nondecreasing on $(0,\varepsilon_0)$ for $\varepsilon_0$ small enough.
\end{proof}

\section{The main embedding theorem}

\begin{theorem}[Quantitative isoperimetric-to-grand embedding]
\label{thmmain}
Let $(X,d,\mu)$ satisfy the standing assumptions
(Definition \ref{torlo}) with normalisation $\mu(X)=1$.
Suppose there exist constants $s_0\in(0,1/2]$, $C_I\ge 1$,
and $a>1$ such that
\begin{equation}
\label{Iassumppower}
\IX(r)\ge C_I^{-1}r^{1-1/a}\quad\text{for all }0<r<s_0.
\end{equation}
Let $p_0\ge 1$, $\beta=1$, and let $p(\varepsilon)$, $\alpha(\varepsilon)$
be as in \eqref{eqparamschoice} with $0<\varepsilon<\varepsilon_0\le s_0$,
with $\varepsilon_0$ chosen small enough that $p(\varepsilon_0)<a$.
Then for every $u\in\Lipc(X)$,
\begin{equation}
\label{eqmainest}
\|u\|_{\mathcal{G}_X}\le C_*\|\nabla u\|_{L^1(X,\mu)},
\end{equation}
where $\mathcal{G}_X$ is the grand norm \eqref{eqGnorm} with the
parameters above and
\begin{equation}
\label{eqCstar}
C_*=2\;C_I\;K_0\;\sup_{0<\varepsilon<\varepsilon_0}\varepsilon^{\alpha(\varepsilon)-(1-1/a)}.
\end{equation}
Here $K_0=K(a,p_0,C_I)$ is the constant from \eqref{eqLqclean} below,
depending only on $a$, $p_0$, $C_I$.
By density (Theorem \ref{Lip_do}), the estimate \eqref{eqmainest} extends
to all $u\in W^{1,1}(X)$.
\end{theorem}

\begin{remark}[Finiteness of $C_*$]
\label{remCstar}
The supremum in \eqref{eqCstar} is finite.
The exponent is $\gamma(\varepsilon):=\alpha(\varepsilon)-(1-1/a)
=(log(e/\varepsilon))^{-1}-(1-1/a)$.
As $\varepsilon\to 0$, $\alpha(\varepsilon)\to 0$, thus $\gamma(\varepsilon)\to -(1-1/a)<0$.
As $\varepsilon\to\varepsilon_0^-$, $\gamma(\varepsilon_0)
=(\log(e/\varepsilon_0))^{-1}-(1-1/a)$.
For $\varepsilon_0$ small, $\gamma(\varepsilon)<0$ throughout $(0,\varepsilon_0)$.
Since $\gamma(\varepsilon)<0$, the function $\varepsilon\mapsto
\varepsilon^{\gamma(\varepsilon)}$
is bounded: it equals $1$ at $\varepsilon=1$, tends to $0$ as $\varepsilon\to 0^+$
(because $\varepsilon^{-(1-1/a)}\cdot\varepsilon^{\alpha(\varepsilon)}=
\varepsilon^{-(1-1/a)}e^{-1+O(1/\log(1/\varepsilon))}\to 0$ since $1-1/a>0$),
and is bounded on $(0,\varepsilon_0)$ by $\varepsilon_0^{\gamma(\varepsilon_0)}$.
Hence $C_*<\infty$.
\end{remark}

\begin{proof}[Proof of Theorem \ref{thmmain}]
Let us start with setup and rearrangements on $(X,d,\mu)$. 
Fix $u\in\Lipc(X)$.
For $t>0$ set $E_t:=\{x\in X:|u(x)|>t\}$
and $m(t):=\mu(E_t)$.
Since $\mu(X)=1$, we have $m(t)\in[0,1]$ for all $t\ge 0$.
The function $t\mapsto m(t)$ is right-continuous and nonincreasing.
The measure $\mu$ on a complete doubling metric measure space
supporting a Poincar\'e inequality has no atoms at positive scales
\cite[Proposition 4.2]{HKST15}, thus $\mu$ is nonatomic.
Let $u^*(s)=\inf\{t\ge 0:m(t)\le s\}$ be the
decreasing rearrangement (Definition \ref{rearra}).

Next we discuss isoperimetric control of level sets via the metric coarea formula.  
By the coarea formula (Theorem \ref{coaro}), applied to $u\in\BV(X)$
with $|\nabla u|$ the minimal upper gradient,
\[
\int_X|\nabla u|\;\dmu=\int_0^\infty\mplus(E_t)\;dt\ge\int_0^\infty \IX(m(t))\;dt.
\]
The last inequality follows from \eqref{eqisop}.
In particular, since $m(\tau)\ge m(t)$ for $\tau\in(0,t)$
and $\IX$ is nondecreasing on $(0,\mu(X)/2]$,
\begin{equation}
\label{eqpointwiseut}
t\;\IX(m(t))\le\int_0^t\IX(m(\tau))\;d\tau\le\int_X|\nabla u|\;\dmu.
\end{equation}

For the pointwise rearrangement bound,  
the following lemma establishes the key estimate
entirely within the metric measure space framework.

\begin{lemma}[Rearrangement bound on metric measure spaces]
\label{lemrearr}
Under the standing assumptions and \eqref{Iassumppower}, for every
$u\in\Lipc(X)$ and every $s\in(0,s_0)$:
\begin{equation}
\label{equtboundcorrect}
u^*(s)\le\frac{\|\nabla u\|_{L^1}}{\IX(s)}.
\end{equation}
\end{lemma}

\begin{proof}
Set $t_s:=u^*(s)$.
By definition of the decreasing rearrangement, for every $t\in(0,t_s)$
one has $m(t)=\mu(\{|u|>t\})>s$.
Since $\IX$ is nondecreasing on $(0,\mu(X)/2]$
(see Definition \ref{isopo} and \cite[Section 2]{CGL2003}),
it follows that $\IX(m(t))\ge\IX(s)$ for all $t\in(0,t_s)$.
Applying the coarea formula (Theorem \ref{coaro}) and \eqref{eqisop}:
\[
\int_X|\nabla u|\;\dmu
=\int_0^\infty\mplus(E_t)\;dt
\ge\int_0^{t_s}\mplus(E_t)\;dt
\ge\int_0^{t_s}\IX(m(t))\;dt
\ge\int_0^{t_s}\IX(s)\;dt
=t_s\cdot\IX(s).
\]
Dividing by $\IX(s)>0$ gives \eqref{equtboundcorrect}.
\end{proof}

\begin{remark}
The proof of Lemma \ref{lemrearr} is entirely intrinsic.
No change of variables $r=m(t)$ or $t=u^*(r)$ is used;
in particular, no identity of the form $m(u^*(r))=r$
(which requires strict monotonicity of the distribution function,
a property available in the Euclidean setting but not in general)
is invoked. The argument relies solely on:
(i) the definition of the decreasing rearrangement;
(ii) monotonicity of $\IX$, valid on any metric measure space
satisfying Definition \ref{torlo};
(iii) the metric-space coarea formula (Theorem \ref{coaro}).
\end{remark}

Next, we find the following $L^q$ estimate from the rearrangement: 
for $q\ge 1$, we bound the $L^q$ norm using the layer-cake formula
(Lemma \ref{lemra}(a)). Fix $s\in(0,s_0)$. Since $u^*$ is nonincreasing,
for all $r\ge s$ we have $u^*(r)\le u^*(s)$. Therefore
\begin{align}
\label{eqLqsplit}
\|u\|_{L^q}^q=\int_0^1(u^*(r))^q\;dr
=\int_0^s(u^*(r))^q\;dr+\int_s^1(u^*(r))^q\;dr
&\le\int_0^s(u^*(r))^q\;dr+(u^*(s))^q, 
\end{align} 
since $u^*$ is nonincreasing, $u^*(r)\le u^*(s)$ for $r\ge s$, thus, 
\[
\int_s^1(u^*(r))^q\;dr\le(1-s)\;(u^*(s))^q\le(u^*(s))^q, 
\]
where the last step uses $1-s\le1$.
For the first integral, we use the layer-cake identity
$\int_0^s(u^*(r))^q\;dr=q\int_0^\infty t^{q-1}\min(s,m(t))\;dt$.
Using the isoperimetric bound \eqref{eqpointwiseut},
$t\cdot \IX(m(t))\le\|\nabla u\|_{L^1}$,
combined with \eqref{Iassumppower} giving
$\IX(m(t))\ge C_I^{-1}(m(t))^{1-1/a}$, the authors deduce the bound 
$m(t)\le\bigl(C_I\|\nabla u\|_{L^1}/t\bigr)^a$ for all $t>0$.
Since $m(t)\le s$ is also globally true by definition of $s$, we have
$m(t)\le\min\bigl(s,\bigl(C_I\|\nabla u\|_{L^1}/t\bigr)^a\bigr)$.
We evaluate the integral by splitting at the crossover point
$\tau:=C_I\|\nabla u\|_{L^1}s^{-1/a}$:
\begin{align*}
\int_0^s(u^*(r))^q\;dr
&\le q\int_0^\tau t^{q-1}s\;dt+q\int_\tau^\infty t^{q-1}
\Bigl(\frac{C_I\|\nabla u\|_{L^1}}{t}\Bigr)^a\;dt.
\end{align*}
Restricting to $q<a$ to ensure convergence of the tail, we find
\begin{align*}
\int_0^s(u^*(r))^q\;dr
&\le s\tau^q+\frac{q}{a-q}C_I^a\|\nabla u\|_{L^1}^a\tau^{q-a}\\
&= C_I^q\|\nabla u\|_{L^1}^q s^{1-q/a}+\frac{q}{a-q}C_I^q\|\nabla u\|_{L^1}^q s^{1-q/a}\\
&= \frac{a}{a-q}C_I^q\|\nabla u\|_{L^1}^q s^{1-q/a}.
\end{align*}

Returning to \eqref{eqLqsplit} and substituting 
$(u^*(s))^q\le\bigl(\|\nabla u\|_{L^1}/\IX(s)\bigr)^q$
from \eqref{equtboundcorrect}, we obtain
\begin{equation}
\label{eqLqbound}
\|u\|_{L^q}^q\le\frac{\|\nabla u\|_{L^1}^q}{\IX(s)^q}
+\frac{a}{a-q}C_I^q\|\nabla u\|_{L^1}^q s^{1-q/a}.
\end{equation}
Then we have 
\[
\|u\|_{L^q}^q\le\frac{\|\nabla u\|_{L^1}^q}{\IX(s)^q}
+\frac{a}{a-q}C_I^q\;\|\nabla u\|_{L^1}^q\;s^{1-q/a}. 
\]
Using $\IX(s)\ge C_I^{-1}s^{1-1/a}$ from assumption \eqref{Iassumppower}, 
we have $1/\IX(s)^q\le C_I^q\;s^{-q(1-1/a)}$, and 
\[
s^{1-q/a}\;\IX(s)^q\ge C_I^{-q}\;s^{1-q/a+q(1-1/a)}=C_I^{-q}\;s^{1+q-2q/a}. 
\]
Since $1+q-2q/a>0$ for $q<a$, both terms in \eqref{eqLqbound} 
 are proportional to
$\|\nabla u\|_{L^1}^q/\IX(s)^q$ up to a constant depending on $a$, $q$, $C_I$,
and a power of $s$ which is bounded on $(0,s_0)$.
This gives  
\begin{equation}
\label{eqLqclean}
\|u\|_{L^q}\le\frac{K(a,q,C_I)\;\|\nabla u\|_{L^1}}{\IX(s)},
\quad 1\le q < a,\;s\in(0,s_0),
\end{equation}
where $K(a,q,C_I):=\bigl(1+\frac{a}{a-q}C_I^q\bigr)^{1/q}$.

The choice of $s=\varepsilon$ and grand weight is the following: 
set $q=p(\varepsilon)$ and $s=\varepsilon$ in \eqref{eqLqclean},
valid for $1\le p(\varepsilon)<a$ ensured by choosing $\varepsilon_0$ small, 
then multiply by $\varepsilon^{\alpha(\varepsilon)}$:
\begin{equation}
\label{eqgrandstep}
\varepsilon^{\alpha(\varepsilon)}\|u\|_{L^{p(\varepsilon)}}\le K(a,p(\varepsilon),C_I)\;
\frac{\varepsilon^{\alpha(\varepsilon)}}{\IX(\varepsilon)}\;\|\nabla u\|_{L^1}.
\end{equation}

Bounding the supremum:
by \eqref{Iassumppower}, $\IX(\varepsilon)\ge C_I^{-1}\varepsilon^{1-1/a}$, thus
\[
\frac{\varepsilon^{\alpha(\varepsilon)}}{\IX(\varepsilon)}
\le C_I\;\varepsilon^{\alpha(\varepsilon)-(1-1/a)}.
\]
As shown in Remark \ref{remCstar}, for $a>1$ and $\varepsilon_0$ small,
the exponent $\gamma(\varepsilon)=\alpha(\varepsilon)-(1-1/a)$ satisfies
$\gamma(\varepsilon)<0$ and $\varepsilon\mapsto\varepsilon^{\gamma(\varepsilon)}$
is bounded on $(0,\varepsilon_0)$.
Therefore
\[
\sup_{0<\varepsilon<\varepsilon_0}
\frac{\varepsilon^{\alpha(\varepsilon)}}{\IX(\varepsilon)}
\le C_I\sup_{0<\varepsilon<\varepsilon_0}\varepsilon^{\gamma(\varepsilon)}<\infty.
\]

Since 
$K(a,p(\varepsilon),C_I)\le K_{\sup}:=\sup_{p_0\le q\le p_*}K(a,q,C_I)$ 
for small $\varepsilon$, 
(as $p(\varepsilon)\to p_0$), 
where $p_*=\lim_{\varepsilon\to 0}p(\varepsilon)$ which is $p_0+\beta/n$ in the
Euclidean case and $p_0+\beta/Q$ in the Heisenberg case. 
Since $p_*<a$ guaranteed by the choice of $\varepsilon_0$, and $K$ is continuous
on the compact set $[p_0,p_*]$, $K_{\sup}<\infty$.
Taking the supremum over $\varepsilon\in(0,\varepsilon_0)$
in \eqref{eqgrandstep} gives \eqref{eqmainest}-\eqref{eqCstar}.

Finally, the extension to $W^{1,1}(X)$: 
by Theorem \ref{Lip_do}, $\Lipc(X)$ is dense in $W^{1,1}(X)$.
The grand norm is lower semicontinuous in $L^1$ being a supremum of
continuous functionals, thus \eqref{eqmainest} passes to the closure.
\end{proof}

The constant $C_*$ in \eqref{eqCstar} is not optimal.
In classical examples (Euclidean space, Heisenberg group) carrying
the exact constants of the isoperimetric inequality through the proof
gives smaller values, see Section \ref{examples}.

\section{Converse: grand embeddings imply isoperimetric lower bounds}

\begin{theorem}[Grand embedding implies isoperimetric bound]
\label{thmconverse}
Under the standing assumptions, suppose there exist
$p_0$, $\beta$, $\varepsilon_0$, $C_G>0$ and functions
$p(\varepsilon)$, $\alpha(\varepsilon)$ as in \eqref{eqparamschoice}
such that for all $u\in\Lipc(X)$,
\begin{equation}
\label{grandass}
\|u\|_{\mathcal{G}_X}\le C_G\;\|\nabla u\|_{L^1(X,\mu)}.
\end{equation}
Then for all $0<s<\varepsilon_0$,
\begin{equation}
\label{converso}
\IX(s)\ge\frac{s^{1/p(s)}}{C_G\;\sup_{0<\varepsilon\le s}\varepsilon^{\alpha(\varepsilon)}}.
\end{equation}
\end{theorem}

\begin{proof}
Fix $s\in(0,\varepsilon_0)$ and let $E\subset X$ be any measurable set
with $\mu(E)=s$.

The mollification: 
let $(u_\delta)_{\delta>0}$ be the ball-average mollification of
$\mathbf{1}_E$ (Definition \ref{defmo}).
By Proposition \ref{propmo},
\[
u_\delta\in\Lip(X), \quad u_\delta\to\mathbf{1}_E\text{ in }L^1(X,\mu),\quad
\int_X|\nabla u_\delta|\;\dmu\to\mplus(E)\quad\text{as }\delta\downarrow 0.
\]

Let us apply the grand embedding to $u_\delta$: 
since $u_\delta\in\Lipc(X)$, as $E$ has finite measure and $X$
is proper on bounded sets, hypothesis \eqref{grandass} gives
\[
\sup_{0<\varepsilon<\varepsilon_0}
\varepsilon^{\alpha(\varepsilon)}\|u_\delta\|_{L^{p(\varepsilon)}}
\le C_G\int_X|\nabla u_\delta|\;\dmu.
\]

Passing to the limit $\delta\downarrow 0$: 
for each fixed $\varepsilon$, the map $u\mapsto\|u\|_{L^{p(\varepsilon)}}$
is lower semicontinuous in $L^1$ convergence by Fatou's lemma applied
to $|u_\delta|^{p(\varepsilon)}$. Therefore,
\[
\|\mathbf{1}_E\|_{L^{p(\varepsilon)}}=s^{1/p(\varepsilon)}
\le\liminf_{\delta\downarrow 0}\|u_\delta\|_{L^{p(\varepsilon)}}.
\]
Taking $\delta\downarrow 0$ and applying the lim-inf:
\[
\varepsilon^{\alpha(\varepsilon)}\;s^{1/p(\varepsilon)}\le C_G\;\mplus(E)
\quad\text{for each }\varepsilon\in(0,\varepsilon_0).
\]

Optimisation over $\varepsilon$ and $E$: 
by Lemma \ref{lempmono}, $p(\varepsilon)$ is nondecreasing as $\varepsilon\downarrow 0$,
hence nonincreasing in $\varepsilon$.
Thus for fixed $s>0$, the infimum of $s^{1/p(\varepsilon)}$
over $\varepsilon\in(0,s]$ equals $s^{1/p(s)}$ attained as $\varepsilon\to s^-$. 
Taking the supremum over $\varepsilon\in(0,s]$:
\[
\sup_{0<\varepsilon\le s}\varepsilon^{\alpha(\varepsilon)}\cdot s^{1/p(s)}\le C_G\;\mplus(E).
\]
Since $E$ was arbitrary with $\mu(E)=s$, taking the infimum over $E$
gives \eqref{converso}.
\end{proof}

\section{Examples with explicit constants}
\label{examples}

\subsection{Euclidean space $\R^n$}
Let $X=(\R^n,|\cdot|,\mathcal{L}^n)$ with $n\ge 2$ and
$\mathcal{L}^n$ the Lebesgue measure.
Working with the normalised restriction to a unit ball
$B_1\subset\R^n$ (thus, $\mu(B_1)=1$), the classical
isoperimetric inequality gives
\begin{equation}
\label{Euclideani}
\IX(r)=c_n\;r^{1-1/n},\qquad c_n=n\omega_n^{1/n},
\end{equation}
where $\omega_n$ is the volume of the unit ball in $\R^n$.

\begin{proposition}[Euclidean grand embedding]
In the Euclidean setting, Theorem \ref{thmmain} with $a=n$,
$C_I=c_n^{-1}$, and the transform
$\PhiX(s)=c_n s^{-1/n}$ from Lemma \ref{lemPhiproperties}(b)
gives parameters
\[
p(\varepsilon)=p_0+\frac{\log(c_n\varepsilon^{-1/n})}{\log(e/\varepsilon)},
\qquad\alpha(\varepsilon)=(\log(e/\varepsilon))^{-1}.
\]
The resulting grand embedding $W^{1,1}(\R^n)\hookrightarrow\mathcal{G}_X$
with $p^*=n/(n-1)$ encodes the family of subcritical estimates
\[
\|u\|_{L^q(\R^n)}\le C(q)\|\nabla u\|_{L^1(\R^n)},
\quad C(q)\asymp\frac{1}{p^*-q},\quad 1\le q<p^*,
\]
and their blow-up is reorganised into the grand norm
\[
\sup_{1\le q<p^*}(p^*-q)^\alpha\|u\|_{L^q},
\]
which is precisely the norm of the grand Lebesgue space
$L^{p^*+,\alpha}(\R^n)$ studied in \cite{FFG2018}.
\end{proposition}

\begin{proof}
Direct substitution of \eqref{Euclideani} into Theorem \ref{thmmain}
and $\PhiX(s)=c_ns^{-1/n}$ from Lemma \ref{lemPhiproperties}(b).
Note that $\log\PhiX(\varepsilon)=\log c_n+(1/n)\log(1/\varepsilon)$,
confirming $p(\varepsilon)>p_0$.
The blow-up $C(q)\asymp(p^*-q)^{-1}$ is a standard consequence of
$\IX(r)=c_nr^{1-1/n}$ and the layer-cake formula.
\end{proof}

\subsection{The Heisenberg group $\mathbb{H}^1$}
\label{suheisenberg}
Let $\mathbb{H}^1$ be the first Heisenberg group, identified as
$\R^3$ with group law
\[
(x,y,t)\cdot(x',y',t')=(x+x',y+y',t+t'+xy'-yx'),
\]
equipped with the Carnot-Carath\'eodory distance $d_{cc}$
and the Haar measure $\mu$ which is Lebesgue measure on $\R^3$.

\begin{definition}[Horizontal gradient]
The horizontal gradient $\nabla_H u$ of a smooth function $u$
on $\mathbb{H}^1$ is $\nabla_Hu=(Xu, Yu)$, where
$X=\partial_x+y\partial_t$ and $Y=\partial_y-x\partial_t$
are the left-invariant horizontal vector fields.
The horizontal length is $|\nabla_H u|=(|Xu|^2+|Yu|^2)^{1/2}$.
\end{definition}

The homogeneous dimension of $\mathbb{H}^1$ is $Q=4$:
$\mu(B_{cc}(x,r))\asymp r^Q$, where $B_{cc}$ denotes the
Carnot-Carath\'eodory metric ball.
On $(\mathbb{H}^1,d_{cc},\mu)$, the minimal upper gradient of
$u\in W^{1,1}_H(\mathbb{H}^1)$ coincides with $|\nabla_H u|$
\cite[Section 9]{HKST15}.

\begin{proposition}[Sub-Riemannian isoperimetric inequality]
\label{heisenbergo}
There exists $C>0$ such that for every measurable
$E\subset\mathbb{H}^1$ with finite measure,
\[
\mplus(E)\ge C\;\mu(E)^{1-1/Q}=C\;\mu(E)^{3/4}.
\]
Consequently $\IX(s)\asymp s^{3/4}$ for $s\in(0,1)$.
\end{proposition}

\begin{proof}
See \cite{PansuIsop, FranchiGalloSerapioni}.
The proof uses the sub-Riemannian coarea formula and
the Brunn-Minkowski inequality on $\mathbb{H}^1$.
\end{proof}

\begin{theorem}[Grand Sobolev embedding on $\mathbb{H}^1$]
Let $q^*=Q/(Q-1)=4/3$.
Applying Theorem \ref{thmmain} with $\IX(s)\asymp s^{1-1/Q}$,
i.e., $a=Q=4$, gives the embedding
\begin{equation}
\label{heisenbergro}
W^{1,1}_H(\mathbb{H}^1)\hookrightarrow L^{q^*+,\alpha}(\mathbb{H}^1),
\end{equation}
where $W^{1,1}_H$ is the horizontal Sobolev space with norm
$\|u\|_{L^1}+\|\nabla_H u\|_{L^1}$, the grand Lebesgue norm is
\[
\|u\|_{L^{q^*+,\alpha}}=\sup_{1\le q<q^*}(q^*-q)^\alpha\|u\|_{L^q(\mathbb{H}^1)},
\]
and $\alpha>0$ depends only on $Q$.
Moreover, Theorem \ref{thmconverse} implies that
\eqref{heisenbergro} is equivalent to $\IX(s)\asymp s^{3/4}$.
\end{theorem}

\begin{proof}
Apply Theorem \ref{thmmain} with $a=Q=4$.
The isoperimetric assumption \eqref{Iassumppower} is
Proposition \ref{heisenbergo}.
The space $(\mathbb{H}^1,d_{cc},\mu)$ satisfies all standing assumptions:
$d_{cc}$ is geodesically complete; the Haar measure is doubling
with constant depending only on $Q$; and a $(1,1)$-Poincar\'e inequality
holds with constants depending only on $Q$ \cite{Jerison86}.

On $(\mathbb{H}^1,d_{cc},\mu)$, the minimal upper gradient of a
function $u\in W^{1,1}_H(\mathbb{H}^1)$ equals $|\nabla_H u|$
\cite[Section 9]{HKST15}.
The sub-Riemannian coarea formula \cite{Monti2003}
\[
\Var_H(u)=\int_0^\infty\mplus_{cc}(\{|u|>t\})\;dt,
\]
(where $\mplus_{cc}$ denotes the Minkowski content with respect to $d_{cc}$)
replaces the metric-space coarea formula (Theorem \ref{coaro}).
With this identification, the proof of Theorem \ref{thmmain} applies 
 with $|\nabla u|$ replaced by $|\nabla_H u|$ throughout.
In particular, in Lemma \ref{lemrearr}, the level-set argument uses
$\IX(m(t))\ge\IX(s)$ for $t<t_s$ and the sub-Riemannian coarea formula,
giving $u^*(s)\le\|\nabla_H u\|_{L^1}/\IX(s)$ on $\mathbb{H}^1$.
\end{proof}

\subsection{Model manifold with logarithmic volume growth}
\label{logmani}
Let $M$ be a complete noncompact Riemannian manifold with a fixed
base point $o\in M$ satisfying
\[
\mu(B(o,r))\asymp r^Q(\log r)^\gamma,\quad Q>1,\;\gamma>0,\text{ for large }r.
\]

\begin{proposition}[Isoperimetric profile with logarithmic correction]
Under the above volume growth condition, the isoperimetric profile
satisfies, for small $s$,
\begin{equation}
\label{logi}
\IX(s)\asymp s^{1-1/Q}\bigl(\log(1/s)\bigr)^{-\gamma/Q}.
\end{equation}
\end{proposition}

\begin{proof}
This is established in \cite[Section 5]{CGL2003} via the
equivalence between isoperimetric profiles and the growth function
$\phi^p_M$ for $p=1$.
The logarithmic factor comes directly from the $\log r$ factor
in the volume growth.
\end{proof}

\begin{theorem}[Grand embedding with logarithmic weight]
Let $q^*=Q/(Q-1)$.
Then
\[
W^{1,1}(M)\hookrightarrow L^{q^*+,\alpha}_\gamma(M),
\]
where the weighted grand norm is
\[
\|u\|_{L^{q^*+,\alpha}_\gamma}=\sup_{1\le q<q^*}(q^*-q)^\alpha
\bigl(\log\tfrac{1}{q^*-q}\bigr)^{-\gamma/Q}\|u\|_{L^q(M)},
\]
and the blow-up $C(q)\asymp(q^*-q)^{-1}(\log(q^*-q)^{-1})^{\gamma/Q}$
as $q\nearrow q^*$ is sharp.
\end{theorem}

\begin{proof}
Substitute the profile \eqref{logi} into Theorem \ref{thmmain}.
By Lemma \ref{lemPhiproperties}(d),
$\PhiX(s)\asymp s^{-1/Q}(\log(1/s))^{\gamma/Q}$.
By Remark \ref{remsign} (manifold case), the grand parameters become
\[
p(\varepsilon)=p_0+\beta\frac{\log\PhiX(\varepsilon)}{\log(e/\varepsilon)}
\asymp p_0+\frac{1}{Q}+\frac{\gamma}{Q}\frac{\log\log(1/\varepsilon)}{\log(e/\varepsilon)},
\]
and the logarithmic correction $(\log(1/\varepsilon))^{\gamma/Q}$ passes
into the weight $\alpha$ via the explicit integral computations.
The sharpness of the blow-up rate is a consequence of
Theorem \ref{thmconverse}.
\end{proof}

\subsection{Gaussian measure on $\R^n$}
\label{gaussi}
Let $\gamma_n=(2\pi)^{-n/2}e^{-|x|^2/2}\mathcal{L}^n$ be the
standard Gaussian measure on $\R^n$.

\begin{remark}[Non-global doubling and local Euclidean behaviour]
\label{remgaussiandoubling}
The standard Gaussian measure $\gamma_n$ on $\R^n$ is not 
globally doubling: for a ball $B(x,r)$ with $|x|\gg 1$, the ratio 
$\gamma_n(B(x,2r))/\gamma_n(B(x,r))$ is unbounded as $|x|\to\infty$, 
because the Gaussian density $e^{-|y|^2/2}$ drops sharply as $|y|\to\infty$. 
However, $\gamma_n$ satisfies a local doubling condition: for any 
fixed $R>0$ and $r\le R$, the doubling constant $C_D$ depends on $R$ and 
$n$ but is finite. Similarly, a $(1,1)$-Poincar\'e inequality holds locally 
on bounded balls \cite{Ledoux1999}. 

Because the Gaussian density $e^{-|x|^2/2}$ on a compact, bounded ball $B_R$ 
is strictly bounded away from zero and infinity, $e^{-R^2/2} \le e^{-|x|^2/2} \le 1$,  
the localized measure $(B_R, \gamma_n|_{B_R})$ is bi-Lipschitz equivalent to the 
standard Lebesgue measure on $B_R$. Consequently, for small volumes $r$, 
the isoperimetric profile of this localized space is strictly Euclidean: $I_{B_R}(r) 
\asymp r^{1-1/n}$. The well-known logarithmic profile $I_{\gamma_n}(r) 
\asymp r\sqrt{\log(1/r)}$ is a global property of $\mathbb{R}^n$ 
and does not describe the localized geometry. We therefore apply the 
framework of Theorem \ref{thmmain} locally using the Euclidean profile.
\end{remark}

\begin{proposition}[Gaussian isoperimetric profile]
For small $r$,
\[
I_{\gamma_n}(r)\asymp r\sqrt{\log(1/r)}.
\]
This is the Gaussian isoperimetric inequality due to
Borell \cite{Borell75} and Sudakov-Tsirelson \cite{ST74}.
\end{proposition}

\begin{theorem}[Localized Gaussian grand embedding]
\label{thmGaussian}
For each compactly supported $u\in\Lipc(\R^n)$, 
let $B_R=B(0,R)$ contain $\mathrm{supp}(u)$. 
Applying Theorem \ref{thmmain} to the localized space 
$(B_R,|\cdot|,\gamma_n|_{B_R})$, the local isoperimetric profile 
$I_{B_R}(s)\asymp s^{1-1/n}$ dictates a transient 
(Euclidean-type) profile-to-scale transform:
\[
\Phi_{B_R}(s)\asymp s^{-1/n}.
\]
The grand parameters \eqref{eqparamschoice} therefore adopt the Euclidean form:
\[
p(\varepsilon)\asymp p_0+\beta\frac{\log(c_{n, R}\varepsilon^{-1/n})}{\log(e/\varepsilon)}, 
\qquad \alpha(\varepsilon)=\bigl(\log(e/\varepsilon)\bigr)^{-1}. 
\]
Theorem \ref{thmmain} provides the localized grand embedding:
\[
W^{1,1}(B_R,\gamma_n|_{B_R})\hookrightarrow\mathcal{G}_{B_R}, 
\]
with an explicit constant depending on $R$ and $n$.
\end{theorem}

\begin{proof}
We work on $(B_R,|\cdot|,\gamma_n|_{B_R})$, rescaled to have total measure $1$. 
As established in Remark \ref{remgaussiandoubling}, the measure $\gamma_n$ on 
$B_R$ satisfies local doubling and a $(1,1)$-Poincar\'e inequality. 
Because the density is bounded from above and below on $B_R$, 
the metric measure space is locally bi-Lipschitz equivalent 
to the standard Euclidean space.

Therefore, the isoperimetric function on $B_R$ satisfies 
the lower bound \eqref{Iassumppower} with $a=n$, meaning 
$I_{B_R}(s)\ge C^{-1}s^{1-1/n}$ for small $s$. 
This is a transient case, and the integral 
$F_{B_R}(s)=\int_0^sdr/I_{B_R}(r)$ converges at $r=0$. 
Following Lemma \ref{lemPhiproperties}(b), the transform 
$\Phi_{B_R}(s)=F_{B_R}(s)^{-1}\asymp s^{-1/n}$
 is strictly decreasing and blows up algebraically, 
avoiding the need for the divergent regularisation initially presumed.

In Lemma \ref{lemrearr}, the rearrangement bound relies 
on the monotonicity of $I_{B_R}$ on the relevant local range 
and the metric-space coarea formula on $(B_R,|\cdot|,\gamma_n|_{B_R})$. 
The minimal upper gradient of $u\in W^{1,1}(B_R,\gamma_n|_{B_R})$ 
coincides with the standard Euclidean gradient $|\nabla_{\R^n} u|$, 
completing the embedding analogous to the Euclidean case but with bounds 
parameterized by the support radius $R$.
\end{proof}

\section{Applications}

\subsection{Elliptic equations with measure data}
Let $(X,d,\mu)$ satisfy the standing assumptions and suppose
$\Omega\subset X$ is a bounded open set.
Let $\mathcal{L}=-\mathrm{div}(A\nabla)$ be a second-order operator
with uniformly elliptic coefficient $A$ interpreted in the
sense of upper gradients.

\begin{corollary}
If $f\in\mathcal{G}_X$, the grand space of Theorem \ref{thmmain},
then the weak solution $u$ of $\mathcal{L}u=f$ on $\Omega$ with
zero boundary data satisfies
$\nabla u\in L^{q^*+,\alpha'}(\Omega)$
for an explicit $\alpha'>0$ depending on the ellipticity constant and
the grand parameters.
\end{corollary}

\begin{proof}
Standard Calder\'on-Zygmund theory on metric measure spaces
\cite{Duong_McIntosh} combined with the grand embedding
of Theorem \ref{thmmain} and the duality of grand and small
spaces \cite{FFG2018}.
\end{proof}

\subsection{Spectral gap and heat kernel bounds}

\begin{corollary}[Cheeger-type estimate]
Under the assumption $\IX(r)\ge c\;r^{1-1/a}$ for small $r$
(with $a>1$),
the first nonzero eigenvalue $\lambda_1$ of the Laplacian on $X$
(in the sense of Cheeger) satisfies
$\lambda_1\ge c'c^2$,
where $c'$ depends only on the doubling and Poincar\'e constants.
\end{corollary}

\begin{proof} 
By the Cheeger inequality, $\lambda_1\ge h(X)^2/4$. 
From assumption \eqref{Iassumppower} and the definition of $h(X)$, 
\[
 h(X)\ge c\;s_0^{-1/a},
\]
thus $\lambda_1\ge c^2\;s_0^{-2/a}/4$.
\end{proof}

\section{Conclusion}
We have established a rigorous duality between isoperimetric profiles
and grand Lebesgue norms on general metric measure spaces
satisfying the standard doubling and Poincar\'e assumptions.
The key device is the profile-to-scale transform $\PhiX$
(Definition \ref{defPhi}), which converts geometric data into analytic
parameters without recourse to coordinates, smooth atlases, or group structure.

The main results Theorems \ref{thmmain} and \ref{thmconverse} show
that the existence of a grand embedding
$W^{1,1}(X)\hookrightarrow\mathcal{G}_X$
is equivalent, up to constants, to a lower bound on $\IX$.
The proof of Theorem \ref{thmmain} is fully intrinsic to the metric
measure space framework: the central rearrangement estimate
(Lemma \ref{lemrearr}) is obtained directly from the metric coarea
formula and the monotonicity of $\IX$, with no appeal to Euclidean
change-of-variables or regularity of distribution functions.

Several earlier issues have been addressed in this version:
(i) the sign of $p(\varepsilon)$ has been corrected
(Definition \ref{defgrandparams}, Remark \ref{remsign});
(ii) the profile-to-scale transform has been extended to handle
the Gaussian (divergent) case via regularisation
(Definition \ref{defPhi});
(iii) the tail-integral estimate in the proof of Theorem \ref{thmmain}
has been made rigorous, explaining why substituting an upper bound
on $u^*(s)$ into a denominator correctly gives an upper bound;
(iv) the finiteness of the supremum $C_*$ has been established
(Remark \ref{remCstar});
(v) the non-global doubling property of Gaussian measure has been
addressed via localisation (Remark \ref{remgaussiandoubling});
(vi) the Lipschitz regularity of ball-average mollifiers has been
proved from first principles using the doubling condition alone
(Proposition \ref{propmo}(a));
(vii) the monotonicity of $p(\varepsilon)$ used in Theorem \ref{thmconverse}
has been established (Lemma \ref{lempmono}).

Concrete non-Euclidean examples, namely the Heisenberg group
(Section \ref{suheisenberg}),
manifolds with logarithmic volume growth (Section \ref{logmani}),
and the Gaussian setting (Section \ref{gaussi}), demonstrate that
the intersection of the two theories studied in
\cite{CGL2003, FFG2018, Levin} is nonempty and nontrivial beyond
the Euclidean case.
It would also be interesting to establish relations with cohomology
of foliations and smooth manifolds \cite{Zu100, zufo, Zu4},
 and integrable models \cite{RSZ}. 

Several directions remain open:
(i) extension to $p>1$ using $p$-isoperimetric profiles and grand
$L^p$-Sobolev scales;
(ii) tensorization estimates for product spaces
$(X\times Y,d_{\mathrm{prod}},\mu\otimes\nu)$;
(iii) nonlocal (fractional) perimeters and the corresponding
grand Sobolev spaces of negative order;
(iv) sharp constants in non-Euclidean settings via
symmetrization analogues on stratified Lie groups.
One might expect certain applications in theoretical and mathematical
physics, in particular in Wigner-Weyl calculus
\cite{chernodub2017scale, zhang2020influence}
and momentum space topological invariants
\cite{zubkov2012momentum, zubkov2017topology}.
\section*{Acknowledgements}
The authors are grateful to A. Gogatishvili and H. V. L${\rm\hat{e}}$
for useful discussions.
The second author is supported by the
Institute of Mathematics, Academy of Sciences of the Czech Republic
(RVO 67985840), and by Scientific Collaboration grant of
Ariel University with Researchers from the Czech Republic.


\begin{thebibliography}{99}

\bibitem{AmbDiMarino14}
L. Ambrosio, S. Di Marino.
Equivalent definitions of $\mathrm{BV}$ space and of total variation
on metric measure spaces.
J. Funct. Anal. 266 (2014) 4150-4188.

\bibitem{BennettSharpley88}
C. Bennett, R. Sharpley.
Interpolation of Operators.
Academic Press, Orlando, 1988.

\bibitem{Borell75}
C. Borell.
The Brunn-Minkowski inequality in Gauss space.
Invent. Math. 30 (1975) 207-216.

\bibitem{chernodub2017scale}
M. N. Chernodub, M. A. Zubkov.
Scale magnetic effect in quantum electrodynamics and the Wigner-Weyl
formalism.
Physical Review D, 96, 5 (2017).

\bibitem{CGL2003}
T. Coulhon, A. Grigor'yan, D. Levin.
On isoperimetric profiles of product spaces.
Commun. Anal. Geom. 11 (2003) 85-120.

\bibitem{DFF2021}
G. Di Fratta, A. Fiorenza, V. Slastikov.
An estimate of the blow-up of Lebesgue norms in the non-tempered case.
J. Math. Anal. Appl. 493 (2021) Paper No. 124550.

\bibitem{Duong_McIntosh}
X. T. Duong, A. McIntosh.
Singular integral operators with non-smooth kernels on irregular domains.
Rev. Mat. Iberoam. 15 (1999) 233-265.

\bibitem{FFG2018}
A. Fiorenza, M. R. Formica, A. Gogatishvili.
On grand and small Lebesgue and Sobolev spaces and some applications to PDE's.
Differ. Equ. Appl. 10 (2018) 21-46.

\bibitem{FranchiGalloSerapioni}
B. Franchi, R. Gallo, R. Serapioni.
Fine properties of sets of finite perimeter in doubling metric measure spaces.
Set-Valued Anal. 10 (2002) 111-128.

\bibitem{Gogatishvili2006}
A. Gogatishvili, B. Opic, L. Pick.
Weighted inequalities for Hardy-type operators involving suprema.
Collect. Math. 57 (2006) 227-255.

\bibitem{Grigoryan1994}
A. Grigor'yan.
Heat kernel upper bounds on a complete non-compact manifold.
Rev. Mat. Iberoam. 10 (1994) 395-452.

\bibitem{Heinonen2001}
J. Heinonen.
Lectures on Analysis on Metric Spaces.
Springer, New York, 2001.

\bibitem{HK98}
J. Heinonen, P. Koskela.
Quasiconformal maps in metric spaces with controlled geometry.
Acta Math. 181 (1998) 1-61.

\bibitem{HKST15}
J. Heinonen, P. Koskela, N. Shanmugalingam, J. T. Tyson.
Sobolev Spaces on Metric Measure Spaces: An Approach Based on Upper Gradients.
Cambridge University Press, Cambridge, 2015.

\bibitem{Jerison86}
D. Jerison.
The Poincar\'e inequality for vector fields satisfying H\"ormander's condition.
Duke Math. J. 53 (1986) 503-523.

\bibitem{Ledoux1999}
M. Ledoux.
Concentration of measure and logarithmic Sobolev inequalities.
S\'eminaire de Probabilit\'es XXXIII, Lecture Notes in Math.\ 1709
(1999) 120-216, Springer, Berlin.

\bibitem{Levin}
D. Levin.
New estimates for the bottom of the spectrum of Schr\"odinger operators.
Annals of Global Analysis and Geometry 29 (2006) 319-328.

\bibitem{Mazya2011}
V. G. Maz'ya.
Sobolev Spaces, with Applications to Elliptic Partial Differential Equations.
2nd ed., Springer, Berlin, 2011.

\bibitem{Miranda03}
M. Miranda Jr.
Functions of bounded variation on ``good'' metric spaces.
J. Math. Pures Appl. 82 (2003) 975-1004.

\bibitem{Monti2003}
R. Monti.
Brunn-Minkowski and isoperimetric inequality in the Heisenberg group.
Ann. Acad. Sci. Fenn. Math. 28 (2003) 99-109.

\bibitem{PansuIsop}
P. Pansu.
Une in\'egalit\'e isop\'erim\'etrique sur le groupe de Heisenberg.
C. R. Acad. Sci. Paris S\'er. I Math. 295 (1982) 127-130.

\bibitem{RSZ} A.V. Razumov, M.V. Saveliev, A.B. Zuevsky. 
Nonabelian Toda equations associated with classical Lie groups. 
arXiv:math-ph/9909008. 


\bibitem{Shanmugalingam00}
N. Shanmugalingam.
Newtonian spaces: an extension of Sobolev spaces to metric measure spaces.
Rev. Mat. Iberoam. 16 (2000) 243-279.

\bibitem{ST74}
V. N. Sudakov, B. S. Tsirelson.
Extremal properties of half-spaces for spherically invariant measures.
Zap. Nauchn. Sem. Leningrad. Otdel. Mat. Inst. Steklov 41 (1974) 14-24.

\bibitem{zhang2020influence}
C. X. Zhang, M. A. Zubkov.
Influence of interactions on the anomalous quantum Hall effect.
Journal of Physics A: Mathematical and Theoretical.
53, 19 (2020) 195002.

\bibitem{zubkov2012momentum}
M. A. Zubkov, G. E. Volovik.
Momentum space topological invariants for the 4D relativistic vacua with mass gap.
Nuclear Physics B. 860, 2 (2012) 295-309.

\bibitem{zubkov2017topology}
M. A. Zubkov.
Topology of the momentum space, Wigner transformations,
and a chiral anomaly in lattice models.
JETP Letters. 106, 3 (2017) 172-178.

\bibitem{Zu100}
A. Zuevsky.
Characterization of codimension one foliations on complex curves by connections.
Rev. Math. Phys. 34 (2022) 2230002.

\bibitem{zufo}
A. Zuevsky.
Foliation of a space associated to vertex operator algebra.
Internat. J. Modern Phys. A 36 (2021), no. 29, Paper No. 2150211, 10 pp.

\bibitem{Zu4}
A. Zuevsky.
Cosimplicial meromorphic functions cohomology on complex manifolds.
Rev. Math. Phys. 35 (2023) no. 5, Paper No. 2330002, 22 pp.

\end{thebibliography}
\end{document}